\documentclass[11pt]{amsart}

\usepackage{ytableau}

\usepackage{amsmath, amsxtra, amsthm, amssymb,mathtools,bm,amsfonts}
\usepackage{mathrsfs}
\usepackage[normalem]{ulem}
\usepackage[mathscr]{euscript}
\usepackage{graphicx}
\usepackage{url}
\usepackage{color}
\usepackage{bbm}
\usepackage{tikz-cd}
\usepackage{tikz}

\usepackage[T1]{fontenc}
\graphicspath{/Figures/}

\usepackage{enumitem}
\usepackage{comment}
\usepackage[pdftex,hidelinks,backref=page]{hyperref}
\hypersetup{
    colorlinks,
    citecolor=magenta,
    filecolor=magenta,
    linkcolor=blue,
    urlcolor=black
}
\usepackage{cleveref}

\usepackage{xcolor}

\renewcommand{\emptyset}{\varnothing}
\newcommand{\NN}{\mathbb N}

\newcommand{\ZZ}{\mathbb Z}

\theoremstyle{definition}
\newtheorem{thm}{Theorem}[section]
\newtheorem{cor}[thm]{Corollary}
\newtheorem{lem}[thm]{Lemma}

\newtheorem{prop}[thm]{Proposition}
\newtheorem{defn}[thm]{Definition}

\newtheorem{eg}[thm]{Example}
\newtheorem{rem}[thm]{Remark}

\newtheorem{fact}[thm]{Fact}

\numberwithin{equation}{section}
\newcommand{\forestpoly}[1]{\mathfrak{P}_{#1}} 
\newcommand{\grovepoly}[1]{\widetilde{\mathfrak{P}}_{#1}} 
\newcommand{\grovepolyb}[1]{\widetilde{\mathfrak{P}}_{#1}^{(\beta)}} 

\newcommand{\qsym}[1]{\operatorname{QSym}_{#1}} 

\newcommand{\internal}[1]{\operatorname{IN}(#1)} 
\newcommand{\suchthat}{\;|\;}

\newcommand{\poly}{\operatorname{Pol}} 
\newcommand{\schub}[1]{\mathfrak{S}_{#1}} 
\newcommand{\groth}[1]{\mathfrak{G}_{#1}} 

\newcommand{\red}[1]{\operatorname{Red}(#1)} 
\newcommand{\des}[1]{\operatorname{Des}(#1)} 

\newcommand{\idem}{\operatorname{id}} 
\newcommand{\ct}{\operatorname{ev}_0} 
\newcommand{\qdes}[1]{\operatorname{LTer}(#1)} 
\newcommand{\tope}[1]{\mathsf{T}_{#1}}
\newcommand{\rope}[1]{\mathsf{R}_{#1}} 
\newcommand{\zigzag}[1]
{\mathsf{ZigZag}_{#1}}

\newcommand{\ltfor}[2][] 
{
{\ifx&#1&%
  {\mathsf{LTFor}_{#2}}
\else
  {\mathsf{LTFor}_{#2}^{#1}}
\fi}
}

\newcommand{\rtfor}[2][] 
{
{\ifx&#1&%
  {\mathsf{RTFor}_{>#2}}
\else
  {\mathsf{RTFor}_{>#2}^{#1}}
\fi}
}

\newcommand{\suppfor}[2][] 
{
{\ifx&#1&%
  {\mathsf{Forest}_{#2}}
\else
  {\mathsf{For}_{#2}^{#1}}
\fi}
}

\newcommand{\fl}[1]{\mathrm{Fl}_{#1}}

\newcommand{\qscoinv}[2][]{
{\ifx&#1&%
  {\operatorname{QSCoinv}_{#2}}
\else
  {{}^{#1}\!\operatorname{QSCoinv}_{#2}}
\fi}
}

\definecolor{ao}{rgb}{0.0, 0.5, 0.0}

\newcommand{\qfl}{\mathrm{QFl}} 

\renewcommand\emph[1]{\textcolor{blue}{\textit{#1}}} 

\newcommand{\LT}{\mathrm{LTer}}
\newcommand{\alpx}{\mathbf{x}}

\newcommand{\indfor}{\operatorname{For}}
\newcommand{\compatible}[1]{\operatorname{Comp}(#1)}
\newcommand{\kompatible}[1]{\operatorname{Komp}(#1)}
\newcommand{\compatiblesub}[2]{\operatorname{Komp}^{(#1)}(#2)}
\newcommand{\ktope}[1]{\mathsf{G}_{#1}}

\newcommand{\cpi}{\hat{\pi}}

\newcommand{\ktopeL}[1]{{\hat{\mathsf{G}}^L_{#1}}}
\newcommand{\ktopeR}[1]{{\hat{\mathsf{G}}^R_{#1}}}
\newcommand{\ktopeH}[1]{{\hat{\mathsf{G}}_{#1}}}

\newcommand{\ktopebH}[1]{{\hat{\mathsf{G}}^{(\beta)}_{#1}}}

\title[]{}

\newcommand{\xl}{{\bm{x}}}

\newcommand{\gl}[1]{\mathrm{GL}_{#1}}

\title{Grove polynomials and $K$-theoretic quasisymmetry}

\author{Philippe Nadeau}
\address{Université Lyon 1, Centrale Lyon, INSA Lyon, Université Jean Monnet, CNRS, ICJ UMR5208, 69622 Villeurbanne, France}
\email{\href{mailto:nadeau@math.univ-lyon1.fr}{nadeau@math.univ-lyon1.fr}}

\author{Hunter Spink}
\address{Department of Mathematics,
University of Toronto, Toronto, ON M5S 2E4, Canada}
\email{\href{mailto:hunter.spink@utoronto.ca}{hunter.spink@utoronto.ca}}

\author{Vasu Tewari}
\address{Department of Mathematical and Computational Sciences, University of Toronto Mississauga, Mississauga, ON L5L 1C6, Canada}
\email{\href{mailto:vasu.tewari@utoronto.ca}{vasu.tewari@utoronto.ca}}

\thanks{
PN was partially supported by French ANR grant ANR-19-CE48-0011 (COMBIN\'E). 
HS and VT acknowledge the support of the NSERC, respectively [RGPIN-2024-04181] and [RGPIN-2024-05433].}

\begin{document}

\begin{abstract}
    We define the grove polynomials, a set-valued extension of forest polynomials.  
    We show that they are $K$-theoretically dual to the quasisymmetric Schubert cells which pave the quasisymmetric flag variety, in the same way that Grothendieck polynomials are dual to Schubert cells in the complete flag variety.  
    As a consequence, the finite truncations of the multi-fundamental quasisymmetric functions of 
    Lam--Pylyavskyy acquire a geometric interpretation as $K$-theoretic representatives of quasisymmetric Schubert cells.
\end{abstract}

\maketitle

\section{Introduction}

Schubert polynomials $\schub{w}$ \cite{LS82} and Grothendieck polynomials 
$\groth{w}^{(\beta)}$ \cite{FKgroth93,LSgroth82} each solve two parallel problems. 
On the combinatorial side, they are dual to composites of the divided difference
operators $\partial_i$ and isobaric divided difference operators $\pi_i^{(\beta)}$, 
which give natural representations of the nil-Hecke and $0$-Hecke algebras respectively. 
On the geometric side, for $B\subset \gl{n}$ the subgroup of invertible upper triangular 
matrices, $\schub{w}$ and $\groth{w}^{(-1)}$ are dual to the Schubert cycles 
$X^w=\overline{BwB/B}$ and structure sheaves of Schubert cells $\mathring{X}^w=BwB/B$ 
in the Bruhat decomposition
$$\gl{n}/B=\bigsqcup_{w\in S_n}\mathring{X}^w, \qquad \mathring{X}^w\cong 
\mathbb{A}^{\ell(w)},$$
under the degree pairing in cohomology and the Euler characteristic pairing in $K$-theory 
respectively \cite{KM05}.

The forest polynomials $\forestpoly{F}$ \cite{NST_1,NT24} solve two analogous problems 
in a quasisymmetric setting. On the combinatorial side, they are dual to composites of 
quasisymmetric divided difference operators $\tope{i}$, which give natural 
representations of the Thompson monoid governing compositions of indexed forests. On the 
geometric side, the quasisymmetric flag variety $\qfl_n\subset \gl{n}/B$ 
\cite{BGNST2} admits an affine paving by quasisymmetric Schubert cells,
$$\qfl_n=\bigsqcup_{F\in \indfor_n}\mathring{X}(F), \qquad \mathring{X}(F)\cong 
\mathbb{A}^{|F|},$$
and the forest polynomials are dual to the closures $X(F)$ under the degree 
pairing. 
The variety $\qfl_n$ is obtained from $\gl{n}/B$ by setting extremal 
Pl\"ucker coordinates to zero outside the $c$-noncrossing partitions for 
$c=s_{n-1}\cdots s_1$. 
Each cell $\mathring{X}(F)$ is a distinguished coordinate 
subspace of a Schubert cell $BwB/B$ associated to a $c$-noncrossing partition $w$, 
with weights determined by a $c$-cluster.

In this paper we introduce the grove polynomials $\grovepoly{F}^{(\beta)}$, a family of polynomials defined via set-valued fillings that simultaneously generalize forest polynomials and extend the multifundamental quasisymmetric polynomials of Lam--Pylyavskyy \cite{LaPy07} to a basis of $\mathbb{Z}[x_1, \ldots, x_n]$. 
Their combinatorial properties are governed by a new family of operators $\ktope{i}$ satisfying the Thompson monoid relations, which play the same role for grove polynomials that the $\pi_i^{(\beta)}$ play for Grothendieck polynomials. 
In particular, this operator calculus yields positivity results analogous to those in the classical theory: the grove polynomials have positive multiplicative structure constants, and Grothendieck 
polynomials expand positively into grove polynomials,
$$
\groth{w}^{(\beta)}=\sum_{\ell(w)\le |F|} \beta^{|F|-\ell(w)}a_w^F\, \grovepolyb{F}\text{ with }a_w^F\ge 0.
$$

As in the classical theory, the geometrically meaningful specialization is $\beta = -1$, 
where $\groth{w}^{(-1)}$ represents the structure sheaf of a Schubert cell in $K$-theory. 
The grove polynomials play the same role for $\qfl_n$: the $\grovepoly{F}^{(-1)}$ are 
Kronecker dual to the structure sheaves $[\mathcal{O}_{\mathring{X}(F)}]\in 
K^\bullet(\qfl_n)$ under the Euler characteristic pairing, in direct analogy with the 
duality between $\groth{w}^{(-1)}$ and $[\mathcal{O}_{\mathring{X}^w}]\in 
K^\bullet(\fl{n})$ \cite{FL94}. Combined with the positivity of the Grothendieck-to-grove 
expansion, this duality implies that the coefficients $a^F_w$ govern the expansion of 
structure sheaves of quasisymmetric Schubert cells into those of classical Schubert cells,
$$[\mathcal{O}_{\mathring{X}(F)}]=\sum_{\ell(w)\le |F|} 
(-1)^{|F|-\ell(w)}a^F_w[\mathcal{O}_{\mathring{X}^w}],$$
which is the appropriate notion of $K$-theoretic positivity.
We note that other bases that refine Grothendieck polynomials, such as the glide 
polynomials \cite{PS19} and Lascoux polynomials \cite{La01,SY23}, do not currently 
have a geometric interpretation of this kind.

Finally, the fundamental quasisymmetric polynomials of Gessel \cite{Ges84} are the special case of forest polynomials indexed by \emph{zigzag trees}, and are thus natural cohomological objects from the perspective of $\qfl_n$. 
The multi-fundamental quasisymmetric polynomials of  Lam--Pylyavskyy are the grove polynomials indexed by these same trees, so our results give precise geometric meaning to the hope expressed in \cite[Introduction]{LaPy07} that the multi-fundamentals are ``$K$-theoretic analogues'' of the fundamental quasisymmetric polynomials.
$$\begin{tikzcd}\text{Forest polynomials $\forestpoly{F}$}\ar[r,dashed]&\text{Grove polynomials $\grovepoly{F}$}\\\text{Fundamental quasisymmetric polynomials $F_\alpha$}\ar[u,hook]\ar[r,dashed]&\text{Multi-fundamental polynomials $\grovepoly{\alpha}$}\ar[u,hook]\end{tikzcd}$$

\section{Background}

Throughout we let $[n]=\{1,2,\ldots,n\}$ for any nonnegative integer $n$.  
By $\delta_{a,b}$ we mean $1$ if $a=b$ and $0$ otherwise, and by $\delta_{a}$ we mean $1$ if $a$ is the identity element and $0$ otherwise. More generally, we use $\delta_P$ for any condition $P$ to mean $1$ if $P$ is satisfied and $0$ otherwise. 
We refer the reader to standard references and surveys \cite{AF24,Br05,BGP,Fulton,PS20,WY23} for further background on the geometric and combinatorial aspects of the classical objects appearing in this article.

We work in the polynomial ring $\ZZ[\xl]\coloneqq \ZZ[x_1,x_2,\ldots]$ in infinitely many variables, on which the infinite symmetric group $S_{\infty}=\langle s_1,s_2,\ldots\rangle$, with $s_i=(i,i+1)$, acts by permuting variables: $\sigma\cdot f(x_1,x_2,\ldots)=f(x_{\sigma(1)},x_{\sigma(2)},\ldots)$. For $w\in S_\infty$, the \emph{length} $\ell(w)$ is the minimal number of simple transpositions needed to express $w$; any such minimal expression $w=s_{i_1}\cdots s_{i_k}$ is called a \emph{reduced word} for $w$, and we write $i_1\cdots i_k\in\red{w}$.

\subsection{Divided differences}

We let $\beta\in \mathbb{Z}$ be a parameter. 
Our K-theoretic combinatorial notions will involve this parameter, but our combinatorial results will not depend on it, so we will often suppress it from the notation when it is not relevant. If we are working with a specific value of $\beta$ we shall state it explicitly.

\begin{defn}
The \emph{divided difference} and \emph{isobaric divided difference} operators on $\ZZ[\xl]$ are
$$\partial_i=\frac{\idem-s_i}{x_i-x_{i+1}}\qquad\text{and}\qquad\pi_i=\partial_i(1+ \beta x_{i+1})=\frac{(1+\beta x_{i+1})\idem -(1+\beta x_i)s_i}{x_i-x_{i+1}}.$$
\end{defn}
These operators satisfy the relations of the nil-Hecke and $0$-Hecke algebras respectively:
\begin{equation}\label{eq:divided_diff_relations}
\begin{alignedat}{3}
    \partial_i^2&=0\qquad &\partial_i\partial_j&=\partial_j\partial_i\text{ if }|i-j|\ge 2\qquad \partial_i\partial_{i+1}\partial_i&&=\partial_{i+1}\partial_i\partial_{i+1}\\ 
    \pi_i^2&=-\beta\pi_i\qquad &\pi_i\pi_j&=\pi_j\pi_i\text{ if }|i-j|\ge 2\qquad \pi_i\pi_{i+1}\pi_i&&=\pi_{i+1}\pi_i\pi_{i+1}
\end{alignedat}
\end{equation}
We note that taking $\beta=0$ in $\pi_i$ we obtain $\partial_i$. 
The relations in~\eqref{eq:divided_diff_relations} imply that for any $w\in S_{\infty}\coloneqq \bigcup_{n\geq 0}S_n$ there are well defined operators $\partial_w\coloneqq \partial_{i_1}\cdots \partial_{i_k}$ and $\pi_w\coloneqq \pi_{i_1}\cdots \pi_{i_k}$ for any choice of reduced word $i_1\cdots i_k \in \red{w}$.

The divided difference operators interact with two important families of polynomials. The \emph{Schubert polynomials} $\schub{w}$ and \emph{Grothendieck polynomials} $\groth{w}$ are the unique families in $\ZZ[\xl]$ satisfying $\ct\schub{w}= \ct\groth{w}=\delta_{w,\idem}$, where $\ct:\mathbb{Z}[\xl]\to \mathbb{Z}$ is the constant term map, and
$$\partial_i\schub{w}=
\begin{cases}
\schub{ws_i}&\ell(ws_i)<\ell(w)\\0&\ell(ws_i)>\ell(w)\end{cases}\qquad\qquad \pi_i\groth{w}=\begin{cases}\groth{ws_i}&\ell(ws_i)<\ell(w)\\-\beta \groth{w}&\ell(ws_i)>\ell(w).
\end{cases}$$

The Schubert polynomials are homogeneous of degree $\ell(w)$. 
Furthermore the Grothendieck polynomial $\groth{w}$ agrees with $\schub{w}$ in lowest degree. Both families are bases of $\mathbb{Z}[\xl]$, and are constructed via the ansatz $\schub{w_{0,n}}=\groth{w_{0,n}}=x_1^{n-1}\cdots x_{n-1}^1$ for the longest element $w_{0,n}\in S_n=\langle s_1,\ldots,s_{n-1}\rangle$.

Define $\cpi_i=\beta\idem+\pi_i$. The operators $-\cpi_i$ satisfy the $0$-Hecke relations, so $\cpi_w\coloneqq \cpi_{i_1}\cdots \cpi_{i_k}$ is well defined for any reduced word $i_1\cdots i_k\in \red{w}$. The two families are then dual to the $\partial_w$ and $\cpi_w$ operators in the sense that

\begin{align}
\label{eq:schubert_grothendieck_extractors}
    \ct \partial_w\schub{w}=\ct \cpi_w\groth{w}=\delta_{w,\idem}.
\end{align}

\subsection{Forest polynomials}
We now introduce the combinatorial objects and operators that will be relevant to the quasisymmetric setting drawing from our previous works \cite{NST_1, NT24}.
\begin{defn}
The \emph{Bergeron--Sottile operators} and \emph{quasisymmetric divided difference operators} \cite{NST_1} are defined on $\ZZ[\xl]$ by
\begin{align*}
    \rope{i}f &= f(x_1,\dots,x_{i-1},0,x_i,x_{i+1},\dots), \\
    \tope{i}f &= \frac{1}{x_i}\bigl(\rope{i+1}f - \rope{i}f\bigr).
\end{align*}
\end{defn}

These satisfy
$$
\rope{i}\rope{j} = \rope{j}\rope{i+1} \quad (i \ge j), \qquad
\tope{i}\tope{j} = \tope{j}\tope{i+1} \quad (i > j).
$$
As explained in \cite{NST_1}, the relations for the $\tope{i}$ are exactly the Thompson monoid relations.
In particular, the composites of these operators are best described by the combinatorics of indexed forests, which we now introduce.

An \emph{indexed forest} is a sequence $F = (T_1, T_2, \ldots)$ of binary trees with all but finitely many equal to the trivial tree $\ast$. We write $\internal{F} = \bigcup \internal{T_i}$ for the set of non-leaf nodes. The size $|F|$ is given by the cardinality $\#\internal{F}$. 

\begin{figure}[!ht]
    \centering
    \includegraphics[width=0.75\linewidth]{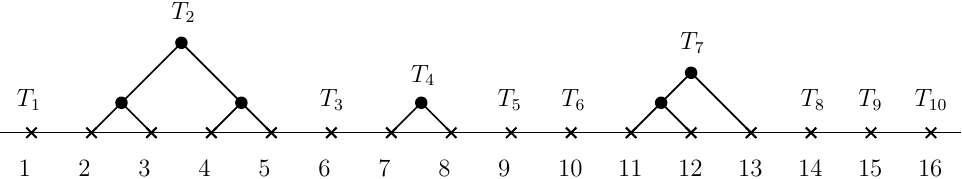}
    \caption{An indexed forest of size 6}
    \label{fig:indexed_forest_example}
\end{figure}
\

For $v\in \internal{F}$, we denote the left and right children by $v_L$ and $v_R$ respectively.
For $v\in \internal{F}\sqcup \NN$ we define $\rho(v)$ by letting it equal $\rho(v_L)$ if $v\in \internal{F}$, and otherwise setting it equal to the leaf label itself.
The roots and leaves of $F$ carry a natural ordering, and $\indfor$ is a monoid under the product $F \cdot G$ formed by attaching the $i$th leaf of $F$ to the $i$th root of $G$, with identity $\emptyset = \ast\ast\cdots$. Let $\indfor_n\subset \indfor$ be the set of indexed forests all of whose nontrivial trees have leaves in $\{1,\ldots,n\}$; the example in \Cref{fig:indexed_forest_example} belongs to $\indfor_n$ for $n\geq 13$. For $\wedge$ the size $1$ binary tree, the assignment
$$
i \mapsto \underbrace{\ast\cdots\ast}_{i-1}\wedge\ast\ast\cdots
$$
is an isomorphism between the Thompson monoid and $\indfor$. In particular, for each $F \in \indfor$ and any \emph{trimming sequence} $i_1 \cdots i_k$ of $F$, the composite $\tope{F} = \tope{i_1}\cdots\tope{i_k}$ is well defined. Under this identification, $F = G \cdot i$ if and only if the $i$th leaf is the left child of a terminal node. We write $\qdes{F} \subset \mathbb{N}$ for the set of all such $i$, and for $i \in \qdes{F}$ write $F/i$ for the unique forest satisfying $F = (F/i) \cdot i$.

\begin{defn}
    A \emph{labeling} of $F\in\indfor$ is a map $\kappa$ from $\internal{F}$ to $\NN$.
    We say that $\kappa$ is \emph{compatible} with $F$ if the following conditions hold:
    \begin{itemize}
        \item $\kappa(v)\leq  \kappa(v_L)$, and 
        \item $ \kappa(v)<  \kappa(v_R)$.
    \end{itemize}
    We define the \emph{weight} of $\kappa$ to be the monomial $\alpx_{\kappa}$ defined as $\alpx_{\kappa}=\prod_{v\in \internal{F}}x_{\kappa(v)}$.
        
\end{defn}

\begin{rem}[Conventions for leaves]
\label{rem:leaf_convention}
When a child $v_L$ or $v_R$ of an internal node is a leaf $\ell\in\NN$, we interpret its label set as the singleton $\kappa(\ell)=\ell$.
\end{rem}

\begin{defn}
\label{defn:not_set_valued}
    Let $F\in \indfor$. Let $\compatible{F}$ denote the set of compatible labelings of $F$.
    We define the forest polynomial $\forestpoly{F}$ by
    \[
        \forestpoly{F}=\sum_{\kappa \in \compatible{F}}\alpx_{\kappa}.
    \]
\end{defn}
A less combinatorial but more useful definition emphasizes their interaction with the $\tope{i}$ operators and establishes the parallel with Schubert polynomials.
The family $(\forestpoly{F})_{F\in\indfor} $ is the unique family satisfying $\ct\forestpoly{F} = \delta_{F,\emptyset}$ and
$$
\tope{i}\forestpoly{F} = \begin{cases} \forestpoly{F/i} & i \in \qdes{F}, \\ 0 & i \notin \qdes{F}. \end{cases}
$$
They are dual to the operators $\tope{F}$ in the sense that $\ct\tope{F}\forestpoly{G} = \delta_{F,G}$.

\section{Grove polynomials}

We now introduce an analogue of the $\pi_i$ operations.

\begin{defn}
 For any $i\geq 1$, let \emph{$\ktope{i}$} be the operator on $\mathbb{Z}[\xl]$ defined by $\ktope{i}=\tope{i}(1+\beta x_{i+1})$.
\end{defn}
Note the equivalent representations
$$\ktope{i}=\rope{i}\pi_i=\rope{i+1}\pi_i=\rope{i}\partial_i(1+\beta x_{i+1})=\tope{i}-\beta \rope{i}.$$

In this section we show the following theorem.
\begin{thm}
\label{thm:GroveCharacterization}
    There is a unique family of polynomials $\grovepoly{F}\in \mathbb{Z}[\xl]$ we call the \emph{grove polynomials}, which satisfies $\grovepoly{\emptyset}=1$, $\ct \grovepoly{F}=0$ if $F\neq \emptyset$, and
$$
\ktope{i}\grovepoly{F}=\begin{cases}\grovepoly{F/i}&i\in \qdes{F}\\ -\beta \rope{i}\grovepoly{F}&i\not\in \qdes{F}.\end{cases}
$$
\end{thm}
\begin{rem}
\label{rem:Riorplus1}
    We have $\ker(\ktope{i}+\beta \rope{i})=\ker(\tope{i})=\ker((1+\beta x_i)\tope{i})=\ker(\ktope{i}+\beta \rope{i+1})$ so $\rope{i+1}$ can be used instead of $\rope{i}$ in the second case above.
\end{rem}

\subsection{Definition of grove polynomials}

Our definition of grove polynomials is combinatorial. 
It is a set-valued extension of the one for forest polynomials \cite{NT24}, and is motivated by and generalizes Lam--Pylyavskyy's construction of multi-fundamental quasisymmetric polynomials \cite{LaPy07}.

\begin{defn}\label{def:kompatible}
    A \emph{set-valued labeling} of $F\in\indfor$ is a map $\kappa$ from $\internal{F}$ to nonempty finite subsets of $\NN$.
    We say that $\kappa$ is compatible with $F$ the following conditions hold:
    \begin{itemize}
        \item $\max \kappa(v)\leq \min \kappa(v_L)$, and 
        \item $\max \kappa(v)<  \min \kappa(v_R)$.
    \end{itemize}
    We define $|\kappa|=\sum_{v\in \internal{F}}|\kappa(v)|$, and define the \emph{weight} of $\kappa$ to be the monomial $\alpx_{\kappa}$ defined as 
    \[
        \alpx_{\kappa}=\prod_{v\in \internal{F}}\prod_{j\in \kappa(v)}x_j.
    \]
\end{defn}

Figure~\ref{fig:kompatible_filling} shows an example of a set-valued filling $\kappa$ compatible with the underlying indexed forest $F$.
\begin{figure}[!ht]
    \centering
    \includegraphics[scale=1]{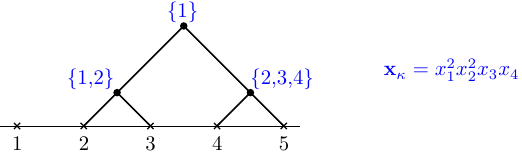}
    \caption{A set-valued labeling compatible with $F$ with weight $x_1^2x_2^2x_3x_4$.}
    \label{fig:kompatible_filling}
\end{figure}

With Definition~\ref{def:kompatible} in hand we are ready to define grove polynomials.

\begin{defn}
\label{defn:setvalued}
    Let $F\in \indfor$. 
    Let $\kompatible{F}$ denote the set of compatible set-valued labelings of $F$.
    We define the grove polynomial $\grovepoly{F}$ by
    \[
        \grovepoly{F}=\sum_{\kappa \in \kompatible{F}}\beta^{|\kappa|-|F|}\alpx_{\kappa}.
    \]
\end{defn}
Note that by definition only nonnegative powers of $\beta$ occur.

\begin{eg}
\label{eg:grove_polynomial}
    Let $F\in \indfor$ be the indexed forest whose sole nontrivial constituent is the tree $T_2$ in Figure~\ref{fig:indexed_forest_example}. Then $\grovepoly{F}$ equals
    \begin{multline*}
         \beta^3 x_1^{2} x_2^{2} x_3 x_4 
         + \beta^2 (2 x_1 x_2^{2} x_3 x_4 +
         2x_1^{2} x_2 x_3 x_4 +
          x_1^{2} x_2^{2} x_4)
         + \beta (x_2^{2} x_3 x_4 +
         x_1 x_2 x_3 x_4 +
         x_1^{2} x_3 x_4 +
         2x_1 x_2^{2} x_4 + 
         2x_1^{2} x_2 x_4\\ + 
         x_1^{2} x_2^{2} x_3+ 
         2x_1 x_2^{2} x_3 + 
         2x_1^{2} x_2 x_3 +
         x_1^{2} x_2^{2})+          x_2^{2} x_4  + 
         x_1 x_2 x_4 + x_1^{2} x_4 + 
         x_2^{2} x_3 + 
         + x_1 x_2 x_3 + 
         x_1^{2} x_3 + 
         x_1 x_2^{2} +
         x_1^{2} x_2.
    \end{multline*}
\end{eg}

\subsection{Proof of \Cref{thm:GroveCharacterization}}
It suffices to show everything for $\beta=1$.
First, we show that the defining relations for grove polynomials have at most one solution.
\begin{lem}
\label{lem:unqiqueifexists}
    There is at most one family of polynomials $(P_F)_{F\in\indfor}$ satisfying $P_\emptyset=1$, $\ct P_{F}=0$ for $F\neq\emptyset$, and for any $F,i$
$$\ktope{i}P_{F}=\begin{cases}P_{F/i}&i\in \qdes{F}\\ -\rope{i}P_{F}&i\not\in \qdes{F}.\end{cases}$$
\end{lem}
\begin{proof}
For the sake of contradiction, assume $P_F,Q_F$ are two families satisfying these conditions. 
Writing $f_F=P_F-Q_F$, we want to prove that $f_F=0$ for all $F\in\indfor$. 
We proceed by induction on $|F|$, starting with  $f_\emptyset=1-1=0$. 

Let $F\neq \emptyset$. 
By induction we have $f_{F/i}=0$ for any $i\in \qdes{F}$. Set $A\coloneqq \qdes{F}$.
Subtracting the relations for $P_F$ and $Q_F$, and recalling $\ktope{i}=\tope{i}-\rope{i}$ and $x_i\tope{i}=\rope{i+1}-\rope{i}$, we obtain
\[
\rope{i+1}f_F=
\begin{cases}
(1+x_i)\rope{i}f_F &i\in A\\ 
\rope{i}f_F &i\not\in A.
\end{cases}
\]
Applying these iteratively for $i=N,N-1,\ldots,1$ for $N$ sufficiently large we have
$$f_F=\rope{N}f_F=\left(\prod_{i\in A}(1+x_i)\right)\rope{1}f_F.$$
Iterating this identity recursively with the $f_F$ on the right hand side, and observing that $\rope{1}^kf_F=\ct(f_F)$ for $k$ sufficiently large, we obtain
$$f_F=\ct(f_F)\prod_{i\in A}\prod_{j=1}^i(1+x_j).$$
The right hand side vanishes since $\ct(f_F)=\ct(P_F)-\ct(Q_F)=0$, concluding the proof.
\end{proof}
 
Now we show that the combinatorial model in Definition~\ref{defn:setvalued} has the correct interaction with the $\ktope{i}$ operators.

\begin{proof}[Proof of \Cref{thm:GroveCharacterization}]
We work with $\beta=1$. The combinatorial definition of grove polynomials ensures that they belong to $\ZZ[\xl]$ and $\grovepoly{\emptyset}=1$, $\ct \grovepoly{F}=0$ if $F\neq \emptyset$. By \Cref{lem:unqiqueifexists} it suffices to show that the polynomials in Definition~\ref{defn:setvalued} satisfy the desired relations. These relations are equivalent to the following.
\begin{align}
\label{eq:grove_trimming}
\tope{i}\grovepoly{F}=\left\lbrace\begin{array}{ll}
0 & i\notin \LT(F),\\
\grovepoly{F/i}+\rope{i}\grovepoly{F} & i\in \LT(F).
\end{array}\right.
\end{align}

Let $\sigma_i$ be the unique increasing bijection $\NN\setminus\{i\}\to \NN$, and $f_i=\sigma_{i+1}^{-1}\sigma_i$ be the unique increasing bijection $\NN\setminus\{i\}\to \NN\setminus\{i+1\}$.
Both maps act naturally on finite sets $S\subseteq\NN$ with $i\notin S$, and on labelings  $\kappa$ with $i\notin\kappa(v)$ for all $v$ by $(\sigma_i\kappa)(v)=\sigma_i(\kappa(v))$ and $(f_i\kappa)(v)=f_i(\kappa(v)).$

Let
\[
  \compatiblesub{i}{F}\coloneqq \{\kappa\in\kompatible{F}: i\notin\kappa(v)\text{ for all }v\in\internal{F}\}.
\]
If $\kappa\in\compatiblesub{i}{F}$, then $\rope{i}(\alpx_{\kappa})=\alpx_{\sigma_i\kappa}$, whereas $\rope{i}(\alpx_{\kappa})=0$ if $i\in\kappa(v)$ for some $v$.
Therefore
\begin{equation}
\label{eq:Ri_as_shift_sum}
  \rope{i}\grovepoly{F}=\sum_{\kappa\in\compatiblesub{i}{F}} \alpx_{\sigma_i\kappa},
  \qquad
  \rope{i+1}\grovepoly{F}=\sum_{\kappa\in\compatiblesub{i+1}{F}} \alpx_{\sigma_{i+1}\kappa}.
\end{equation}

Since $f_i$ is strictly increasing on its domain, it preserves all weak/strict inequalities, so $f_i\kappa\in\compatiblesub{i+1}{F}$ whenever $\kappa\in\compatiblesub{i}{F}$. This gives an injection $\compatiblesub{i}{F}\to \compatiblesub{i+1}{F}$ which preserves the weights in \eqref{eq:Ri_as_shift_sum} since $\sigma_{i+1}(f_i\kappa)=\sigma_i\kappa$ by definition. It follows
\begin{equation}
\label{eq:pairing_reduction}
  \tope{i}\grovepoly{F}\,=\,\frac{1}{x_i}\left(\rope{i+1}\grovepoly{F}-\rope{i}\grovepoly{F}\right)
  \,=\,\frac{1}{x_i} 
  \sum_{\kappa\in\mathcal{G}_i(F)} \alpx_{\sigma_{i+1}\kappa},
\end{equation}
where $\mathcal{G}_i(F)$ consists of those labelings in $\compatiblesub{i+1}{F}$ that are not of the form $f_i\kappa$ for $\kappa\in \compatiblesub{i}{F}$. These are exactly labelings $\kappa\in\compatiblesub{i+1}{F}$ with a node $v\in\internal{F}$ with $\rho_F(v)=i$  such that $i\in\kappa(v)$. By replacing $v$ repeatedly with its left child $v_L$ if necessary, one can assume that $v_L$ is a leaf. This then forces $v_R$ to be a leaf as well, since the inequalities $\max \kappa(v)=i<\min \kappa(v_R)<i+1=\rho_F(v_R)$ cannot be satisfied if $v_R\in\internal{F}$. 

In summary, $\mathcal{G}_i(F)$ is the set of $\kappa\in\compatiblesub{i+1}{F}$ such that there is a terminal node $u$ in $F$ with leaf-children $i$ and $i+1$, and moreover $i=\max \kappa(u)$. If there is no such node $u$, $\mathcal{G}_i(F)$ is empty and thus \eqref{eq:pairing_reduction} gives us the first case of \eqref{eq:grove_trimming}. Otherwise we distinguish two cases:

\begin{itemize}
\item  If $\kappa(u)=\{i\}$, then consider the labeling of $F/i$ obtained by deleting the node $u$, and then applying $\sigma_{i+1}$ to the remaining labels. This establishes a bijection between compatible labelings of $F$ with $\kappa(u)=\{i\}$ and $\kompatible{F/i}$.
The contribution to~\eqref{eq:pairing_reduction} in this case is exactly $\grovepoly{F/i}$.
\item Now suppose $\kappa(u)=S\sqcup \{i\}$ with $S\neq \emptyset$.
Construct the labeling $\kappa^{-}$ from $\kappa$ by removing the element $i$ from $\kappa(u)$, and subsequently apply $f_i^{-1}$, i.e. replace every occurrence of $i$ in $\kappa^{-}$ by $i+1$.
Clearly this labeling is compatible and avoids $i$ by construction; in fact we get a bijection to $\compatiblesub{i}{F}$. The contribution to~\eqref{eq:pairing_reduction} is exactly
$\rope{i}\grovepoly{F}$.
\end{itemize}

The two cases taken together imply the second case in \eqref{eq:grove_trimming} and conclude the proof.
\end{proof}

\section{Grove Extractors and applications to positivity}

The operators $\cpi_i$ act as extractors for the basis of Grothendieck polynomials, in the sense that they can be used to pick coefficients in that basis (see~\eqref{eq:schubert_grothendieck_extractors}). The analogue for grove polynomials is more complicated -- we define the pair of operations
\begin{align}\label{eq:grove_extractors}
\ktopeL{i}&=\ktope{i}+\beta\rope{i}=\tope{i},\\
\ktopeR{i}&=\ktope{i}+\beta \rope{i+1}=(1+\beta x_i)\tope{i}.\end{align}
By Theorem~\ref{thm:GroveCharacterization} and the remark following it, these satisfy the formulas
\begin{align}\label{eq:to_be_used_later}
\ktopeL{i}\,\grovepoly{F}=
\begin{cases}\grovepoly{F/i}+\rope{i}\grovepoly{F}&i\in \qdes{F}\\0&i\not\in \qdes{F}
\end{cases}
\quad\text{ and }\quad\ktopeR{i}\,\grovepoly{F}=
\begin{cases}\grovepoly{F/i}+\rope{i+1}\grovepoly{F}&i\in \qdes{F}\\0&i\not\in \qdes{F}.
\end{cases}
\end{align}

The composition of operators $\ktopeR{i},\ktopeL{i}$ together with $\rope{i}$ are best understood via a \textit{graphical calculus}. Let $\poly_1=\ZZ[x]$, so we have the natural identification $\ZZ[\alpx]=\poly_1^{\otimes \infty}$. Define $\tope{}:\poly_1^{\otimes 2}\to \poly_1$ by $\tope{}{f(x,y)}=(f(x,0)-f(0,x))/x$, and $\rope{}:\poly_1\to \poly_0$ by $\rope{}(f(x))=f(0)$. Also let $\ktopeL{}=\tope{}$ and $\ktopeR{}=(1+\beta y)\tope{}$. Then notice that for any $O\in\{\ktopeL{},\ktopeR{},\rope{}\}$,
$$
O_{i}=\idem^{\otimes i-1}\otimes\; O\otimes \idem^{\otimes \infty}.
$$
These satisfy universal commutation relations, as do any shifts of given fixed operators $\poly_1^{\otimes 2}\to \poly_1$ and $\poly_1\to \poly_0$, which come from acting on disjoint sets of variables.  
The composition of such operators up to these commutations can be encoded as ``marked nested forests'' (see \cite[Remark 3.17]{nst_c}) defined and illustrated in Figure~\ref{fig:marked_forest_example}:
\begin{itemize}
    \item we allow trees nested under each other, in addition to the sequential order,
    \item internal nodes have to be labelled to indicate which $\poly_1^{\otimes 2}\to \poly_1$ operator is used (in our case, $L,R$ will stand for $\ktopeL{},\ktopeR{}$),
    \item tree roots may carry an extra marking $\times$ (and must carry such a marking if the root is nested below another tree) to indicate a one-to-zero operation (in our case we have just $\rope{}$, so we will need no extra labeling).
    \end{itemize}  

\begin{figure}[!ht]
    \centering
    \includegraphics[width=0.7\linewidth]{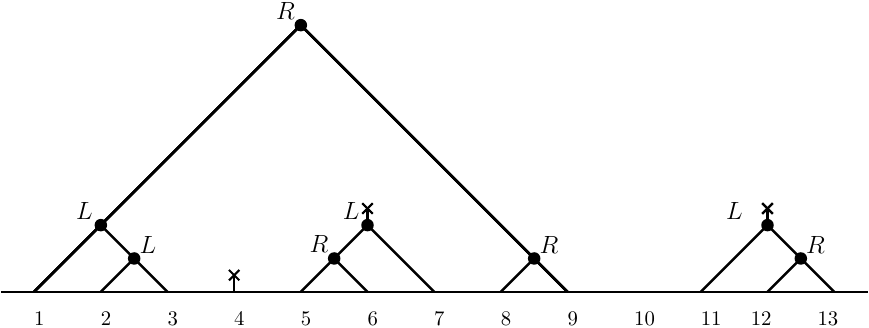}
    \caption{Marked forest for  $\ktopeR{1}\ktopeL{1}\ktopeL{2}\rope{4}\rope{5}\ktopeL{5}\ktopeR{5}\ktopeR{8}\rope{11}\ktopeL{11}\ktopeR{12}$}
    \label{fig:marked_forest_example}
\end{figure}

It follows in particular that there is a naturally defined operator $\ktopeH{F}$, the \emph{grove extractor}, defined recursively by $\ktopeH{\emptyset}=\idem$ and for any $i\in \qdes{F}$
$$\ktopeH{F}=\begin{cases}\ktopeH{F/i}\ktopeR{i}&i\text{ is a right child in $F$}\\\ktopeH{F/i}\ktopeL{i}&\text{ otherwise.}\end{cases}$$
We can  represent it by an indexed forest with labels $R$ on right children and $L$ on left children and roots, see Figure~\ref{fig:G_forest_example}.
\begin{figure}[!ht]
    \centering
    \includegraphics[width=0.5\linewidth]{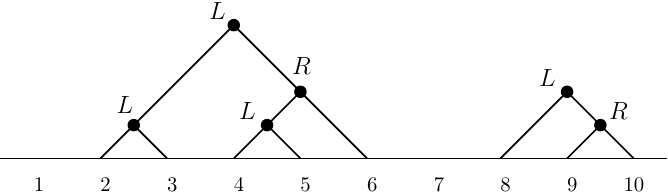}
    \caption{Labeled forest $F$ for  $\ktopeH{F}=\ktopeL{2}\ktopeL{2}\ktopeR{4}\ktopeL{4}\ktopeL{8}\ktopeR{9}$}
    \label{fig:G_forest_example}
\end{figure}

In this section we show the following theorem.

\begin{thm}\label{thm:groveextractors}
    The grove polynomials are dual to the $\ktopeH{F}$ operators in the sense that
$$\ct \ktopeH{F}\,\grovepoly{G}=\delta_{F,G}.$$
\end{thm}
\begin{eg}
    For the $F$ in Example~\ref{eg:grove_polynomial} we have $\ktopeH{F}=\ktopeH{2}^{L}\,\ktopeH{2}^{L}\,\ktopeH{4}^{R}=\ktopeH{2}^{L}\,\ktopeH{3}^{R}\,\ktopeH{2}^{L}$.
    One checks that
    \[
        \grovepoly{F} \xrightarrow{\ktopeH{2}^L} \grovepoly{2\cdot 3}+\grovepoly{1\cdot1\cdot3}\xrightarrow{\ktopeH{3}^R}  \grovepoly{2}+\grovepoly{2\cdot3}+\grovepoly{1\cdot 1}+\grovepoly{1\cdot1\cdot3}\xrightarrow{\ktopeH{2}^L} \grovepoly{\emptyset}+\grovepoly{1}\xrightarrow{\ct}1,
    \]
    where each summand has its indexed forest written as its Thompson monoid representative.
\end{eg}

In the remainder of the section we work with $\beta=1$. 
To do this, we first record identities which help expand $\ct\ktopeH{F}\rope{i}$ as a nonnegative linear combination of $\ct\ktopeH{F'}$. 
These are immediately verified, so we omit the proofs.
\begin{lem}\leavevmode
\label{lem:relations}
 The identities in Figure~\ref{fig:extractor_relations} are true. In symbols, if $A,B\in \{L,R\}$ then
 \begin{align*}\ktopeH{i}^A\rope{i+1}\ktopeH{i+1}^B&=\rope{i+1}\ktopeH{i}^A\ktopeR{i+1}+\rope{i}\ktopeH{i+1}^A\ktopeL{i+1}+\rope{i}\rope{i+2}\ktopeH{i+1}^A,\\
\ktopeH{i}^A\rope{i+1}&=\rope{i}\ktopeH{i+1}^A+\rope{i+1}\ktopeH{i}^A.
\end{align*}
    
    \begin{figure}[!ht]
    \centering
    \includegraphics[width=0.8\textwidth]{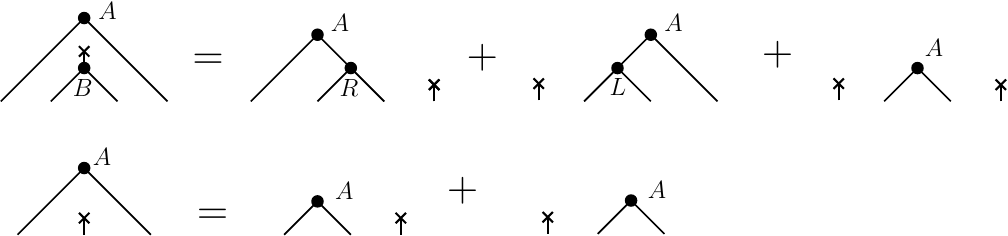}
    \caption{Relations between $\ktopeL{i}$, $\ktopeR{i}$, and $\rope{i}$ depicted graphically. These represent a region of a marked nested forest included in a topological disk.
    }
    \label{fig:extractor_relations}
\end{figure} 
\end{lem}

\begin{cor}
\label{cor:Straightencor}Any operator of the form $\ktopeH{F}\rope{i}$ for $F\in\indfor$ can be rewritten as a linear combination of operators $\rope{i_1}\cdots \rope{i_t}\ktopeH{G}$ with $i_j$ integers and $G\in\indfor$. In particular $\ct\ktopeH{F}\rope{i}$ can be rewritten as a linear combination of $\ct \ktopeH{G}$.
\end{cor}

\begin{proof}
We apply the relations in Lemma~\ref{lem:relations} from left to right to nested marked forests, only applying the bottom relation when the marking has no nontrivial tree below it. The top relation preserves the property that trees are correctly $L/R$ colored, while the bottom relation when applied preserves this property if $\times$ is the root of a trivial tree. These relations always decrease the total nesting depth of $\times$ markings, so this process eventually terminates.
When these relations cannot be applied, then the resulting forests correspond to operators as described. 
\end{proof}

We can now prove the main result of this section.

\begin{proof}[Proof of \Cref{thm:groveextractors}]
    We work with $\beta=1$. We induct on $|F|$, the result being true for $F$ empty. Let $i\in \qdes{F}$. If $i\not\in \qdes{G}$ then with $\ktopeH{i}\in \{\ktopeL{i},\ktopeR{i}\}=\{\tope{i},(1+x_i)\tope{i}\}$ we have
    $$\ct \ktopeH{F/i}\ktopeH{i}\grovepoly{G}=0$$
    since $\tope{i}\grovepoly{G}=0$.

    If $i\in \qdes{G}$, we distinguish three cases for $i$.  
    
    Assume first that $i$ is a right child in $F$, and let $v\in \internal{F}$ be its parent. Then
    $$\ct\ktopeH{F}\grovepoly{G}=\ct \ktopeH{F/i}\ktopeR{i}\grovepoly{G}=\ct \ktopeH{F/i}(\grovepoly{G/i}+\rope{i+1}\grovepoly{G}).$$
    Because $\delta_{F/i,G/i}=\delta_{F,G}$, it remains to show that $\ktopeH{F/i}\rope{i+1}\grovepoly{G}=0$. The subtree of $F/i$ rooted at $v$ does not nest the mark $\times$ corresponding to $\rope{i+1}$, thus its combinatorial structure is not modified as we iterate the relations of Lemma~\ref{lem:relations}. In particular $i$ remains a right child through the process, so all of the resulting $\ct\ktopeH{F'}$  created (all of which have $|F'|<|F|$) do not have $i\in \qdes{F'}$, and hence $F'\ne G$ and $\ct\ktopeH{F'}\grovepoly{G}=0$ by induction.

     Assume now that $i$ is a left child or the root of a tree in $F$. Then we have $$\ct\ktopeH{F}\grovepoly{G}=\ct\ktopeH{F/i}\ktopeL{i}\grovepoly{G}=\ct\ktopeH{F/i}(\grovepoly{G/i}+\rope{i}\grovepoly{G}).$$
    If $i$ is a left child, then one can proceed completely similarly to the case of the right child. Finally, if $i$ is a root then we have a commutation $\ktopeH{F/i}\rope{i}=\rope{i}\ktopeH{F'}$, where $F'$ is obtained by deleting the $\wedge$ associated to $i$ in $F$. So it remains to show that $\ct \ktopeH{F'}\grovepoly{G}=\delta_{F',G}$ is equal to $0$, which is true because $i\not\in \qdes{F'}$ since $i$ is now a trivial tree in $F'$.
\end{proof}

\subsection{The basis of grove polynomials}
\begin{prop}
    The forest polynomial $\forestpoly{F}$ is the lowest degree homogeneous component of the grove polynomial $\grovepoly{F}$.
\end{prop}
\begin{proof}
This follows directly from Definition~\ref{defn:setvalued}.
\end{proof}
\begin{thm}\label{th:grove_expansions}
  Grove polynomials are a basis for $\ZZ[\alpx]$. Any $f\in \mathbb{Z}[\xl]$ can be uniquely written as
  $$f=\sum_{F\in \indfor} (\ct\ktopeH{F}f)\grovepoly{F}.$$
\end{thm}
\begin{proof}
We work with $\beta=1$, the proof for arbitrary $\beta$ is identical.
The existence of grove extractors $\ktopeH{F}$ from Theorem~\ref{thm:groveextractors} shows that the grove polynomials are $\mathbb{Z}$-linearly independent, so it suffices to show that they span. Fix $f\in \mathbb{Z}[x_1,\ldots,x_n]$. Because $\ktopeL{i}=\tope{i}$ decreases degree and $\ktopeR{i}=(1+x_i)\tope{i}$ does not increase degree, any $F$ where the number of left children is $\ge \deg(f)+1$ has $\ktopeH{F}f=0$. For $1\le i \le n$, the operators $\ktopeR{i}$ and $\ktopeL{i}$ preserve $\mathbb{Z}[x_1,\ldots,x_n]$ and for $i>n$ $\ktopeR{i}$ and $\ktopeL{i}$ kill $\mathbb{Z}[x_1,\ldots,x_n]$. There are only finitely many forests with $\qdes{F}\subset \{1,\ldots,n\}$ that have at most $\deg(f)$ many left children, so all but finitely many grove extractors applied to $f$ are zero. We claim that
$$f=\sum_F (\ct \ktopeH{F}f)\grovepoly{F}.$$
By the discussion above the right hand side is well-defined. The difference of the left and right hand side is a polynomial $g$ with the property that $\ct\ktopeH{F}g=0$ for all $F$. We claim that this implies $g=0$. 
Let $d$ be the minimal degree in $g$, and consider any forest $F$ of size $d$. For $L$ the polynomial operator that returns the lowest degree homogenous component, we claim that $$\ct \tope{F}L(f)=\ct L(\ktopeH{F}f).$$
Indeed, $\ktopeL{i}=\tope{i}$ and any $\ktopeR{i}=x_i\tope{i}+\tope{i}$ occurring in $\ktopeH{F}$ can be replaced by a $\tope{i}$ for degree reasons. If we pick $F$ such $\forestpoly{F}$ appearing nontrivially in the forest decomposition of $L(f)$, then we obtain 
$0\ne \ct \tope{F}L(f)=\ct L(\ktopeH{F}f)$, a contradiction.
\end{proof}

\subsection{Positive Expansions}

We call a polynomial $f\in \ZZ[\alpx]$ \emph{grove-positive} if its expansion in the basis of $\beta=1$ grove polynomials has nonnegative integer coefficients, and we say a polynomial operator is \emph{grove-positive} if it takes grove polynomials to nonnegative linear combinations of grove polynomials.

\begin{thm}
    \label{th:positivity_of_composites}
    The operators $\rope{i}$, $\ktopeL{i}$, and $\ktopeR{i}$ are grove-positive.
\end{thm}
\begin{proof}
Given Theorem~\ref{th:grove_expansions}, we want to show $\ct \ktopeH{G}X_{i}\grovepoly{F}\geq 0$, where $X\in \{\rope{},\ktopeL{},\ktopeR{}\}$. 
    We shall establish the more general claim that $\ct\Phi\,\grovepoly{F}\ge 0$ where $\Phi$ is any composite of $\rope{i}$, $\ktopeL{i}$ and $\ktopeR{i}$. 
    
    To this end, whenever we have a $\rope{k}$ as the rightmost operator in $\Phi$ we apply straightening relations \Cref{lem:relations} to move to the left, which is then absorbed into the $\ct$ operator as $\ct\rope{j}=\ct$ for any $j$.
    On the other hand, whenever $\ktopeL{i}$ or $\ktopeR{i}$ is the rightmost operator in $\Phi$, we apply the relations in~\eqref{eq:to_be_used_later}.
    This  procedure eventually terminates and all coefficients produced are nonnegative.
\end{proof}
We record some consequences of Theorem~\ref{th:positivity_of_composites} that are of combinatorial interest.
These generalize results established by the authors in the setting of forest polynomials \cite[Section 7]{NST_1}.
We first describe the interaction between Grothendieck polynomials and grove polynomials.
This result can be viewed in the same vein as the \textit{Grothendieck-to-Lascoux} expansion\footnote{originally conjectured by Reiner--Yong \cite{RY21}} due to Shimozono--Yu \cite{SY23} and the \textit{Grothendieck-to-glide} expansion due to Pechenik--Searles \cite{PS19}; see also \cite{MPS21} for a host of other ``K-theoretic polynomials'' inspired by other combinatorial considerations.

\begin{cor}
    \label{cor:positiveexps_grothendieck}
     Grothendieck polynomials are grove-positive, i.e. we have $$\groth{w}^{(\beta)}=\sum_{\ell(w)\le |F|} \beta^{|F|-\ell(w)}a_w^F \grovepolyb{F}\text{ with }a_w^F\ge 0.$$ 
\end{cor}
\begin{proof}
    The claim follows from \Cref{th:positivity_of_composites} and the relations:
        $$\ktopeL{i}\groth{w}=\begin{cases}0&i\not\in \des{w}\\ \rope{i}\groth{w}+\rope{i}\groth{ws_i}&\text{otherwise}\end{cases}\text{ and }\ktopeR{i}\groth{w}=\begin{cases}0&i\not\in \des{w}\\ \rope{i+1}\groth{w}+\rope{i+1}\groth{ws_i}&\text{otherwise.}\end{cases}\qedhere
        $$    
\end{proof}

Our next result concerns the structure constants of multiplication for grove polynomials. This will use the Leibniz identities for $\ktopeH{i}\in \{\ktopeL{i},\ktopeR{i}\}=\{\tope{i},(1+x_i)\tope{i}\}$ that
\begin{equation}\label{eqn:Leibniz}\ktopeH{i}(fg)=\ktopeH{i}(f)\rope{i}(g)+\rope{i+1}(f)\ktopeH{i}(g).\end{equation}
\begin{cor}
    \label{cor:positiveexps}
The product $\grovepoly{F}\grovepoly{G}$ is grove-positive.
\end{cor}
\begin{proof}
         We want to show that $\ct\ktopeH{H}(\grovepoly{F}\grovepoly{G})\ge 0$. 
         To this end we induct on $|H|$ with the case $|H|=0$ being trivial. 
         For $i\in \qdes{H}$ we have $\ktopeH{H}=\ktopeH{H/i}\ktopeH{i}$ with $\ktopeH{i}\in \{\ktopeL{i},\ktopeR{i}\}$, so by \eqref{eqn:Leibniz} we have
        $$\ct\ktopeH{H}(\grovepoly{F}\grovepoly{G})=\ct\ktopeH{H/i}((\ktopeH{i}\grovepoly{F})(\rope{i}\grovepoly{G})+(\rope{i+1}\grovepoly{F})(\ktopeH{i}\grovepoly{G}))$$
        and the result now follows by Theorem~\ref{th:positivity_of_composites} and induction on $|H|$.
\end{proof}

\begin{rem}\label{rem:necessary}
    Recall that Lenart \cite{Le99} gave a positive expansion for the Grothendieck polynomials in the Schubert basis (the expansion ibid. is signed as $\beta$ is set to $-1$). 
    We can show analogously that grove polynomials expand positively in terms of forest polynomials by using the extractors from \cite[Section 6]{NST_1}. 
    In the opposite direction, Lascoux \cite{La04} established that the expansion of a Schubert polynomial in terms of Grothendieck polynomials is sign-alternating; see also \cite{Wei25} for a uniform perspective on changing bases in terms of bumpless pipe dreams. 
    Empirical evidence suggests a similar phenomena persists when expanding forest polynomials in terms of grove polynomials. We leave it as an open question.
\end{rem}

\section{$K$-theory}

In this section we set $\beta=-1$. Recall that if $B\subset \gl{n}$ are the upper triangular matrices, then $\fl{n}\coloneqq \gl{n}/B$ is the \emph{complete flag variety}, parametrizing complete flags of subspaces
$$\{0\subsetneq V_1\subsetneq \cdots \subsetneq V_{n-1}\subsetneq \mathbb{C}^n\}.$$
The $K$-theory $K^\bullet(\gl{n}/B)$ and cohomology ring $H^\bullet(\gl{n}/B)$ both have presentations \cite{Bor53, Dem74} as the symmetric coinvariants 
$$\mathbb{Z}[x_1,\ldots,x_n]/\langle f(x_1,\ldots,x_n)-f(0,\ldots,0)\suchthat f\in \mathbb{Z}[x_1,\ldots,x_n]^{S_n}\rangle.$$
where in $K$-theory $x_i$ is interpreted as $1-[\mathcal{F}_i/\mathcal{F}_{i-1}]$ and in cohomology $x_i$ is interpreted as $c_1(\mathcal{F}_i/\mathcal{F}_{i-1})$, where $\mathcal{F}_i$ is the tautological quotient flag. Schubert cycles $X^w=\overline{BwB/B}$ are the closures of Schubert cells $\mathring{X}^w=BwB/B$ in the Bruhat decomposition of the complete flag variety
$$\gl{n}/B=\bigsqcup_{w\in W}\mathring{X}^w\text{ with }\mathring{X}^w=BwB/B\cong \mathbb{A}^{\ell(w)}.$$
The Bruhat decomposition relates to $\partial_w$ and $\cpi_w$ via the identities 
$$\deg_{\gl{n}/B}(f\cdot [\overline{BwB}])=\ct\partial_w f\text{ and }\chi_{\gl{n}/B}(f\cdot [\mathcal{O}_{BwB}])=\ct \cpi_w^{(-1)} f=1,$$
where $\chi$ is the Euler characteristic. In particular, the Schubert $\schub{w}$ and Grothendieck specializations $\groth{w}^{(-1)}$ are Kronecker dual to the homology classes $X^w$ and structure sheaves $\mathcal{O}_{\mathring{X}^w}$ respectively. 

In \cite{BGNST2} we geometrized forest polynomials via the \emph{quasisymmetric flag variety} $\qfl_n\subset \gl{n}/B$. 
This has an affine paving as described in \cite[Section 9]{BGNST2}
$$\qfl_n=\bigsqcup_{F\in \indfor_n}\mathring{X}(F),$$
by \emph{quasisymmetric Schubert cells}
whose closures were called the \emph{quasisymmetric Schubert cycles} $X(F)$. 
There is a surjection $H^\bullet(\gl{n}/B)\to H^\bullet(\qfl_n)$ realizing $H^\bullet(\qfl_n)$ as the ring of quasisymmetric coinvariants
$$\mathbb{Z}[x_1,\ldots,x_n]/\langle f(x_1,\ldots,x_n)-f(0,\ldots,0)\suchthat f\in \qsym{n}\rangle,$$
and we have
$\deg_{\qfl_n}(f\cdot X(F))=\ct\tope{F} f,$
so the forest polynomials are Kronecker dual to the quasisymmetric Schubert cycles.

We now show that the grove polynomials are $K$-theoretically dual to the quasisymmetric Schubert cells $\mathring{X}(F)$. 
To do this, we recall the recursive construction for $\mathring{X}(F)$ from \cite[Section 4]{BGNST2}.
\begin{fact}
    Let $p_i:\gl{m}/B\to \gl{m}/P_i$ be the projection which forgets the $i$'th flag. Then for $f\in K^\bullet(\gl{m}/B)$ we have $$p_i^*p_{i,*}(f)=\pi_i^{(-1)} f.$$
\end{fact}
We recall how this implies that $\chi_{\gl{n}/B}(f\cdot [\mathcal{O}_{\mathring{X}^w}])=\ct \cpi_wf.$ The base case is that for $w=\idem$ we have $BwB$ is a single point and $\chi(f\cdot \mathcal{O}_{\{pt\}}))=\ct f$. 
The key fact which we can apply recursively is that in $K^\bullet(\gl{n}/B)$ we have $$p_i^*p_{i,*}[\mathcal{O}_{\mathring{X}^w}]=\begin{cases} [\mathcal{O}_{\mathring{X}^w}]+[\mathcal{O}_{\mathring{X}^{ws_i}}]&\ell(ws_i)>\ell(w)\\
0&\ell(ws_i)<\ell(w),
\end{cases}$$
(see \cite[Lemma 3.8(a)]{ALSS}).
This implies that if $\ell(ws_i)>\ell(w)$ then $[\mathcal{O}_{\mathring{X}^{ws_i}}]=(p_i^*p_{i,*}-\idem)[\mathcal{O}_{\mathring{X}^w}]$ and so 
$$\chi(f\cdot [\mathcal{O}_{\mathring{X}^{ws_i}}])= \chi((\cpi_i^{(-1)}f)\cdot [\mathcal{O}_{\mathring{X}^w}]).$$

We now recall certain pattern maps introduced by Bergeron--Sottile \cite{BS98}; these were crucial to the construction of $\qfl_n$ \cite[Section 3]{BGNST2}.
\begin{defn}
    Define maps $\Psi_i^{\pm}: \gl{n}/B\hookrightarrow \gl{n+1}/B$ by the inclusion of matrices
    \begin{align*}
    \begin{bmatrix}A_{(i-1)\times (i-1)}&B_{(i-1)\times n-(i-1)}\\C_{n-(i-1)\times (i-1)}&D_{n-(i-1)\times n-(i-1)}\end{bmatrix}\xrightarrow{\Psi_i^{-}} \begin{bmatrix}A_{(i-1)\times (i-1)}&0&B_{(i-1)\times n-(i-1)}\\0&1&0\\
    C_{n-(i-1)\times (i-1)}&0&D_{n-(i-1)\times n-(i-1)}\end{bmatrix}\quad  (1\le i \le n)\\
    \begin{bmatrix}A_{(i-1)\times i}&B_{(i-1)\times n-i}\\C_{n-(i-1)\times i}&D_{n-(i-1)\times n-i}\end{bmatrix}\xrightarrow{\Psi_i^+} \begin{bmatrix}A_{(i-1)\times i}&0&B_{(i-1)\times n-i}\\0&1&0\\
    C_{n-(i-1)\times i}&0&D_{n-(i-1)\times n-i}\end{bmatrix}\quad  (1\le i \le n-1).
    \end{align*}
    \end{defn}
\begin{fact}[{\cite[Proposition 4.3]{LRS}}]
    The homological operations $(\Psi_i^{-})_*$ and $(\Psi_i^+)_*$ are adjoint to the cohomological operations $\rope{i},\rope{i+1}$ in the sense that
    $$\chi_{\gl{m+1}/B}((\Psi_i^-)_*f)=\rope{i}\chi_{\gl{m}/B}(f)\text{ and }\chi_{\gl{m+1}/B}((\Psi_i^+)_*f)=\rope{i+1}\chi_{\gl{m}/B}(f).$$
\end{fact}

As we showed in~\cite[Section 3]{BGNST2}, the maps $\Psi_i^{\pm}$ are closed immersions 
and $p_{i,*}$ is an isomorphism when restricted to the image of $\Psi_i^-$. Therefore, for any $Z\subset \gl{m}/B$,
$$
p_{i}^*p_{i,*}(\Psi_i^-)_*[\mathcal{O}_Z]=p_{i}^*[\mathcal{O}_{p_i\Psi_i^- Z}]
=[\mathcal{O}_{p_i^{-1}p_i\Psi_i^- Z}].
$$
By inclusion-exclusion, this identity extends from closed subvarieties to arbitrary constructible sets. 
The cells $\mathring{X}(F)$ are recursively defined by
$$
\mathring{X}(F)=\begin{cases}
\bigl(p_i^{-1}p_i\Psi_i^- \mathring{X}({F/i})\bigr) \setminus \bigl(\Psi_i^+\mathring{X}(F)\bigr)
  & i\text{ is a right child,}\\[4pt]
\bigl(p_i^{-1}p_i\Psi_i^- \mathring{X}({F/i})\bigr) \setminus \bigl(\Psi_i^-\mathring{X}(F)\bigr)
  & \text{otherwise.}
\end{cases}
$$
For $i\in \qdes{F}$, we immediately deduce
$$
\chi\bigl(f\cdot [\mathcal{O}_{\mathring{X}(F)}]\bigr)=\begin{cases}
\chi\bigl((\ktopeR{i,\beta=-1}f)\cdot [\mathcal{O}_{\mathring{X}(F/i)}]\bigr)
  & i\text{ is a right child,}\\[4pt]
\chi\bigl((\ktopeL{i,\beta=-1}f)\cdot [\mathcal{O}_{\mathring{X}(F/i)}]\bigr)
  & \text{otherwise.}
\end{cases}
$$
Applying this recursion and using the definition of $\ktopeH{F}$ yields the following.

\begin{thm}
For any $f\in K^\bullet(\qfl_n)$ and $F\in\indfor_n$,
$$\chi_{\qfl_n}\bigl(f\cdot [\mathcal{O}_{\mathring{X}(F)}]\bigr)=\ct \ktopeH{F,\beta=-1}f.$$
In particular, the grove polynomials $\grovepoly{F}^{(-1)}$ are dual to the structure sheaves of the quasisymmetric Schubert cells.
\end{thm}

As an immediate corollary, we obtain the $K$-theoretic expansion of structure sheaves of quasisymmetric Schubert cells in terms of those of classical Schubert cells.

\begin{cor}
In $K^\bullet(\gl{n}/B)$,
$$[\mathcal{O}_{\mathring{X}(F)}]=\sum_{w\in S_{n}}(-1)^{|F|-\ell(w)}a^F_w
[\mathcal{O}_{\mathring{X}^w}],$$
where $a^F_w$ are the coefficients from \Cref{cor:positiveexps_grothendieck}.
\end{cor}
\begin{proof}
The classes $[\mathcal{O}_{\mathring{X}^w}]$ form a basis of $K^\bullet(\gl{n}/B)$ dual to the functionals $X\mapsto \chi(\groth{w}^{(-1)}\cdot X)$. 
The coefficient of $[\mathcal{O}_{\mathring{X}^w}]$ in $[\mathcal{O}_{\mathring{X}(F)}]$ is therefore 
$\ct\ktopeH{F,\beta=-1}\groth{w}^{(-1)}$, which by the theorem equals the coefficient $(-1)^{|F|-\ell(w)}a^F_w$ of $\grovepoly{F}^{(-1)}$ in the expansion of 
$\groth{w}^{(-1)}$ into $\beta=-1$ grove polynomials.
\end{proof}

\section{Multi-fundamental polynomials}
\label{sec:multifundamentals}

Recall that a permutation $w\in S_{\infty}$ is $n$-Grassmannian if $\des{w}\subset\{n\}$. Such a permutation uniquely determines a partition $\lambda(w)$ with at most $n$ parts by way of its Lehmer code $(c_i)_{i\geq 0}$ where $c_i=|\{j>i\suchthat w(i)>w(j)\}|$.
The subfamilies of Schubert and Grothendieck polynomials obtained by restricting to $n$-Grassmannian permutations recover the Schur and symmetric Grothendieck polynomials in variables $x_1$ through $x_n$ respectively, with the latter representing structure sheaves of Schubert varieties in the K-theory of the Grassmannian \cite{Bu02}.
Analogously, if we specialize the forest polynomials $\forestpoly{F}$ to the \emph{zigzag forests} $\zigzag{n}=\{Z\suchthat \qdes{Z}\subset \{n\}\}$, then we recover Gessel's \cite{Ges84} \emph{fundamental quasisymmetric polynomials} (see \cite{NST_1,NT24}). 
See Figure~\ref{fig:zigzag} for an example of a zigzag forest $Z\in\zigzag{4}$.
We are thus naturally led to consider grove polynomials indexed by forests in $\zigzag{n}$.

\begin{figure}[!ht]
    \includegraphics[scale=0.8]{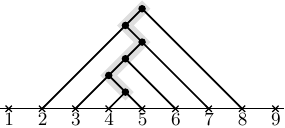}
    \caption{The zigzag forest in $\zigzag{4}$ corresponding to $\alpha=(2,3,1)$}
    \label{fig:zigzag}
\end{figure}

 The \emph{multi-fundamental} quasisymmetric functions $\widetilde{L}_\alpha$, where  $\alpha=(\alpha_1,\ldots,\alpha_t)$ and the $\alpha_i$ are positive integers, were introduced by Lam and Pylyavskyy~\cite[p.15]{LaPy07}. 
 Explicitly,
\begin{align}
\label{eq:multi_fundamental}
\widetilde{L}_\alpha(x_1,x_2,\ldots)=\sum_{(S_1,\ldots,S_k)}\prod_{j=1}^k\prod_{i\in S_j}{x_i}
\end{align}
where $k=\alpha_1+\cdots +\alpha_t$, the $S_i$ are finite subsets of $\NN$ such that $\max S_i\leq \min S_{i+1}$ for all $i<k$, and $\max S_i<\min S_{i+1}$ when $i$ belongs to $\{\alpha_1,\alpha_1+\alpha_2,\ldots,\alpha_1+\cdots+\alpha_{t-1}\}$. 
We are primarily interested in truncating these functions to finitely many variables.
\begin{prop}
    Fix $n\geq 1$. The family of polynomials $\{\grovepoly{Z}\suchthat Z\in \zigzag{n}\}$ coincides with the family $\{\widetilde{L}_\alpha(x_1,x_2,\ldots,x_n)\suchthat \alpha=(\alpha_1,\ldots,\alpha_t)\text{ with }t\leq n\}$.
\end{prop}

\begin{proof}
Given a zigzag tree $Z$ of size $k$, let $\alpha_1$ be the number of vertices in the left branch coming from and including the root. 
Cut this branch, repeat the process, and let $\alpha_2,\alpha_3,\ldots$ denote the number of vertices in the successive left branches. 
Define $\alpha=(\alpha_1,\alpha_2,\ldots,\alpha_t)$. Notice that $t\leq n$ otherwise we cannot have $Z\in\zigzag{n}$.
Then
    \[\widetilde{L}_\alpha(x_1,x_2,\ldots,x_n)=\grovepoly{Z}.\]
Indeed, to match up with the definition in~\eqref{eq:multi_fundamental}, assign the label $S_i$ to the $i$th node starting from the root (starting at $1$). 
This constructs a compatible labeling of $Z$, while the extra leaf condition $\max S_k\leq n$ corresponds to the truncation to $n$ variables in $\widetilde{L}_\alpha$. 
\end{proof}

\begin{thm}\label{th:multifundamentals_expansions}
    At $\beta=1$, the multi-fundamental polynomials $\{\grovepoly{Z}\suchthat Z\in \zigzag{n}\}$ form a basis of $\qsym{n}$. For any $f\in  \qsym{n}$ we have the decomposition
    $$f=\sum_{Z\in \zigzag{n}} (\ct \ktopebH{Z}f)\,\grovepoly{Z}.$$
\end{thm}
\begin{proof}
First, note that $\tope{i}\grovepoly{Z}=0$ for $i\not\in \qdes{F}$ by Theorem~\ref{thm:GroveCharacterization}. This shows that $\grovepoly{Z}\in \qsym{n}$ \cite[Corollary 6.9]{NST_1}.

Now, given an arbitrary $f\in \qsym{n}$ it suffices to show that its grove expansion only contains $\grovepoly{Z}$ with $Z\in \qsym{n}$, which by Theorem~\ref{thm:groveextractors} can be rephrased equivalently as
$\ct \ktopeH{F} f=0$ for $F\not\in \zigzag{n}$. 
For such an $F$, there exists $i\ne n\in \qdes{F}$ by definition. Let $\epsilon=\delta_{\text{$i$ is a right child}}$. Then
$\ktopeH{F}f=\ktopeH{F/i}(1+\epsilon x_i)\tope{i}f=0$, and we conclude.
\end{proof}
\begin{eg}
    Let $f=x_{2}^{2} x_{3} + x_{1} x_{2} x_{3} + x_{1}^{2} x_{3} + x_{1}^{2} x_{2} \in \qsym{4}$.
    We obtain the following expansion of $f$ into multi-fundamental quasisymmetric polynomials:
    \[
    f=\grovepoly{2\cdot2\cdot 3} - \grovepoly{1\cdot2\cdot 2\cdot3}-2\,\grovepoly{1\cdot1\cdot2\cdot3}   
    \]
    We explain the coefficient of $\grovepoly{1\cdot1\cdot2\cdot3}$. 
    To this end we apply the operator $\ct\ktopeH{1\cdot1\cdot2\cdot3}$ to $f$:
    \[
        f\xrightarrow{\ktopeH{3}^R}
        x_{2}^{2} x_{3} + x_{1} x_{2} x_{3} + x_{1}^{2} x_{3} + x_{2}^{2} + x_{1} x_{2} + x_{1}^{2}
        \xrightarrow{\ktopeH{2}^R}
        x_{2}^{2} - x_{1}^{2} x_{2} + x_{1} x_{2} + x_{2} - x_{1}^{2} + x_{1}
        \xrightarrow{\ktopeH{1}^L}
        -2x_1
        \xrightarrow{\ktopeH{1}^L} 
        -2
    \]
    
\end{eg}

\begin{rem}
    A special case of the positivity of grove polynomials into forest polynomials mentioned in Remark~\ref{rem:necessary}  establishes that multi-fundamentals expand positively into fundamentals. This recovers \cite[Theorem 5.12]{LaPy07}.
\end{rem}

\bibliographystyle{plain}
\bibliography{Groves}

\end{document}